\documentclass[11pt]{amsart}
\usepackage[latin1]{inputenc}
\usepackage{amsmath}
\usepackage{amsthm}
\usepackage{amsfonts}
\usepackage{amssymb}
\newtheorem{theorem}{Theorem}[section]
\newtheorem{lemma}[theorem]{Lemma}
\newtheorem{remark}[theorem]{Remark}
\newtheorem{proposition}[theorem]{Proposition}
\newtheorem{corollary}[theorem]{Corollary}
\renewcommand{\theequation}{
\arabic{section}.\arabic{equation}%
}

\newcommand{\re}{\mathbb{R}}
\newcommand{\sob}[1]{W^{1,p}\left (#1\right )}
\newcommand{\norm}[1]{\Vert #1\Vert}
\newcommand{\leb}[2]{L^{#1}\left(#2\right )}
\def\D{\nabla}
\begin{document}
\title[Sharp estimates of radial minimizers of $p$-Laplace equations]{Sharp estimates of radial minimizers of $p$-Laplace equations}
\author{Miguel Angel Navarro \and Salvador Villegas}
\thanks{The authors have been supported by the MEC Spanish grant MTM2012-37960}
\address{Departamento de An\'{a}lisis
Matem\'{a}tico, Universidad de Granada, 18071 Granada, Spain.}
\email{mnavarro\_2@ugr.es \and svillega@ugr.es}
\begin{abstract}
In this paper we study semi-stable, radially symmetric and decreasing solutions $u\in W^{1,p}(B_1)$ of $-\Delta_p u=g(u)$ in $B_1\setminus\{ 0\}$, where $B_1$ is the unit ball of $\mathbb{R}^N$, $p>1$, $\Delta_p$ is the $p-$Laplace operator and $g$ is a general locally Lipschitz function. We establish sharp pointwise estimates for such solutions. As an application of these results, we obtain optimal pointwise estimates for the extremal solution and its derivatives (up to order three) of the equation $-\Delta_p u=\lambda f(u)$, posed in $B_1$, with Dirichlet data $u|_{\partial B_1}=0$, where the nonlinearity $f$ is an increasing $C^1$ function with $f(0)>0$ and
$
\lim_{t\rightarrow+\infty}{\frac{f(t)}{t^{p-1}}}=+\infty.
$
In addition, we provide, for $N\geq p+4p/(p-1)$, a large family of semi-stable radially symmetric and decreasing unbounded $W^{1,p}(B_1)$ solutions.
\end{abstract}
\maketitle
\section{Introduction and main results}
This paper is concerned with the semi-stability of radially symmetric and decreasing solutions $u\in W^{1,p}(B_1)$ of
\begin{equation}
-\Delta_p u=g(u)\text{  in }B_1\setminus \left\lbrace 0\right\rbrace ,
\label{mainequation}
\end{equation}
where $p>1$, $\Delta_p$ is the $p-$Laplace operator, $B_1$ is the unit ball of $\re^N$, and $g:\re\longrightarrow\re$ is a locally Lipschitz function.

By abuse of notation, we write $u(r)$ instead of $u(x)$, where
$r=\vert x\vert$ and $x\in \re^N$. We denote by $u_r$ the radial derivative of a radial function $u$.

A radial solution $u\in H^1(B_1)$ of \eqref{mainequation} such that $u_r(r)<0$ for all $r\in (0,1)$ is
called semi-stable if

$$
\int_{B_1}{(p-1)\vert u_r\vert^{p-2}\vert\xi_r\vert^2-g'(u)\xi^2}\geq 0,
$$ for every radially symmetric function $\xi\in C_c^1\left (B_1\setminus\{0\}\right )$.

Note that the above expression is nothing but the
second variation of the energy functional associated to \eqref{mainequation}:

\begin{equation}\label{functional}
E_\Omega(u):=\frac{1}{p}\int_{\Omega}{\vert \nabla u\vert^p\,dx}-\int_{\Omega}{G(u) \, dx},
\end{equation}
\noindent where $G'=g$ and $\Omega\subset B_1$. Thus, if $u$ is a radial local
minimizer of \eqref{functional} with $\Omega=B_1$ (i.e.,  for every $\delta \in (0,1)$ there exists $\varepsilon_\delta>0$ such that $E_{B_1\setminus\overline{B_\delta}}(u)\leq E_{B_1\setminus\overline{B_\delta}}(u+\xi)$, for all radial functions $\xi\in C_c^1 \left(B_1\setminus\overline{B_\delta}\right)$) satisfying $\Vert \xi \Vert_{C^1}\leq \varepsilon_\delta$), then $u$ is a semi-stable solution of \eqref{mainequation}. Other general situations include semi-stable solutions: for instance, minimal solutions, extremal solutions, and also some solutions between a sub and a supersolution (see \cite[Rem. 1.7]{MR2476421} for more details). All the results obtained in this paper were obtained by the second author in \cite{MR2885956} for the Laplace operator ($p=2$).

As an application of some general results obtained in this paper
for this class of solutions (for arbitrary $g\in C^1(\Bbb{R})$), we consider the following problem

\begin{equation}
\left\lbrace
\begin{array}{rcll}
-\Delta_p u&=&\lambda f(u)&\text{  in }B_1,\\
u&>&0&\text{  in }B_1,\\
u&=&0&\text{  in }\partial B_1,\\
\end{array}
\right.
\label{equation_ext}
\refstepcounter{equation}
\tag{$\theequation_{\lambda,p}$}
\end{equation}
where $\lambda>0$ and $f$ is an increasing $C^1$ function with $f(0)>0$ and
\begin{equation}
\lim_{t\rightarrow+\infty}{\frac{f(t)}{t^{p-1}}}=+\infty.
\label{ext_cond_function}
\end{equation}

This problem is studied by Cabr\'e and Sanch\'on in \cite{MR2276329} for general smooth bounded domains $\Omega$ of  $\mathbb{R}^N$. It is proved that there exists a positive parameter $\lambda^*$ such that if $\lambda\in (0,\lambda^*)$ then \eqref{equation_ext} admits a minimal (smallest) solution $u_\lambda\in C^1(\overline{\Omega})$ and if $\lambda\in (\lambda^*,+\infty)$ then \eqref{equation_ext} admits no regular solution. In addition, for $\lambda\in (0,\lambda^*)$ the minimal solution $u_\lambda$ is semi-stable (in a similar sense of the definition when $\Omega=B_1$). On the other hand, we may consider the increasing limit
$$u^*:=\lim_{\lambda\uparrow\lambda^*}u_\lambda.$$

In the case $p=2$ it is well-known that $u^*$ is a weak solution of \eqref{equation_ext}, for $\lambda=\lambda^*$. It is called the extremal solution. For general $p$, $\Omega$ and $f$, it is not known if $u^*$ is a weak solution of \eqref{equation_ext}, for $\lambda=\lambda^*$. In the case $\Omega=B_1$, Cabr\'e, Capella and Sanch\'on \cite{MR2476421} proved that $u^*$ is actually a semi-stable radially decreasing energy solution (i.e. $u^*\in W_0^{1,p}$) of \eqref{equation_ext}. Hence we can apply to the extremal solution the results obtained in this paper for this kind of solutions.

We refer to \cite{MR2373729,MR2569324} for surveys on minimal and extremal solutions and to \cite{MR1605678,MR2361735,MR2338421,MR2393398,MR2143221,MR2387858,MR2286013,MR0340701,MR1779693,MR2317200,MR3010053} for other interesting results in the topic of extremal solutions.

The main result obtained in \cite{MR2476421} related to the extremal solution of \eqref{equation_ext} is the following

\begin{theorem}(\cite{MR2476421}).\label{cabrecapellasanchon}
Let $\Omega=B_1$, $p>1$, and that $f$ satisfies \eqref{ext_cond_function}. Let $u^\ast$ be the extremal solution of \eqref{equation_ext}. We have that

\begin{enumerate}

\item[i)] If $N<p+4p/(p-1)$, then $u^\ast \in L^\infty (B_1)$.

\

\item[ii)] If $N=p+4p/(p-1)$, then $u^\ast(r)\leq C\, \left\vert \log
r \right\vert $,  $\forall r\in (0,1)$ and for some constant $C$.

\

\item[iii)] If $N>p+4p/(p-1)$, then $\displaystyle{u^\ast(r)\leq C\, r^{-\frac{1}{p}\left(N-2\sqrt{\frac{N-1}{p-1}}-p-2\right)}\left\vert \log r
\right\vert^\frac{1}{p}} \ $,  $\forall r\in (0,1)$ and for some constant $C$.

\

\item[iv)] If $N\geq p+4p/(p-1)$, then $\displaystyle{\vert u_r^\ast(r)\vert\leq C\, r^{-\frac{1}{p}\left(N-2\sqrt{\frac{N-1}{p-1}}-2\right)}\left\vert \log r
\right\vert^\frac{1}{p}} \ $,  $\forall r\in (0,1)$ and for some constant $C$.

\end{enumerate}

\end{theorem}

\

In this paper we establish sharp pointwise
estimates for $u^\ast$ and its derivatives (up to order three). We improve the above theorem, answering
affirmatively to an open question raised in \cite{MR2476421}, about the
removal of the factor $\left\vert \log r\right\vert^\frac{1}{p}$.

\begin{theorem} Let $p>1$. Suppose that $f$ satisfies \eqref{ext_cond_function}. Let $u^*$ be the extremal solution of \eqref{equation_ext}. We have that

\begin{itemize}

\item[i)] If $p\leq N<p+4p/(p-1)$, then $u^*(r)\leq C(1-r)$, $\forall r\in(0,1]$.

\

\item[ii)] If $N=p+4p/(p-1)$, then $u^*(r) \leq C\vert \log r\vert ,\;\forall r\in (0,1].$

\

\item[iii)] If $N>p+4p/(p-1)$, then \[u^*(r)\leq C\left (r^{{-\frac{1}{p}\left(N-2\sqrt{\frac{N-1}{p-1}}-p-2\right)}}-1\right ),\;\forall r\in (0,1].\]

\

\item[iv)] If $N\geq p+4p/(p-1)$, then
\[
\vert\partial_r^{(k)} u^*(r)\vert \leq Cr^{{-\frac{1}{p}\left(N-2\sqrt{\frac{N-1}{p-1}}+(k-1)p-2\right)}},\;\forall r\in (0,1],\,\forall k\in\{1,2\}.
\]
\

\item[v)] If $N\geq p+4p/(p-1)$, and $f$ is convex, then
\[
\vert u_{rrr}^*(r)\vert \leq Cr^{{-\frac{1}{p}\left(N-2\sqrt{\frac{N-1}{p-1}}+2p-2\right)}},\;\forall r\in (0,1].
\]

\end{itemize}
Where $C=C_{N,p}\min\limits_{t\in[1/2,1]}{\vert u_r(t)\vert}$, and $C_{N,p}$ is a constant depending only on $N$ and $p$ .
\label{th_ext}
\end{theorem}

\begin{remark} In \cite{MR1266373} Garc\'{\i}a-Azorero, Peral and Puel proved that if $f(u)=e^u$ and $N=p+4p/(p-1)$ then

$$u^*(r)=-p\log r\mbox{    and    }\lambda^\ast=4p^p/(p-1).$$

This shows that the pointwise estimates of Theorem \ref{th_ext} are optimal for $N=p+4p/(p-1)$.

On the other hand, in \cite{MR2476421} Cabr\'e and Sanch\'on proved that if $N>p+4p/(p-1)$ and $f(u)=(1+u)^m$, where

$$m:=\frac{(p-1)N-2\sqrt{(p-1)(N-1)}-p+2}{N-2\sqrt{\frac{N-1}{p-1}}-p-2},$$

\noindent then

$$u^*(r)=r^{{-\frac{1}{p}\left(N-2\sqrt{\frac{N-1}{p-1}}-p-2\right)}}-1$$

\noindent and

$$\lambda^*=\left( \frac{p}{m-(p-1)}\right)^{p-1}\left( N-\frac{mp}{m-(p-1)}\right).$$

This also shows the optimality of the pointwise estimates of Theorem \ref{th_ext} for the case $N>p+4p/(p-1)$.
\end{remark}

As mentioned before, the proof of Theorem \ref{th_ext} is based on general properties of semi-stable radially decreasing energy solutions. Our main results about these type of solutions are the following.

\begin{theorem}
Let $p>1$, $g:\re\longrightarrow\re$ be a locally Lipschitz function, and $u\in\sob{B_1}$ be a semi-stable radial solution of \eqref{mainequation} satisfying $u_r(r)<0$ for all $r\in (0,1)$. Then there exists a constant $C_{N,p}$ depending only on $N$ and $p$ such that:
\begin{itemize}

\item[i)] If $p\leq N<p+4p/(p-1)$, then $\norm{u}_{\leb{\infty}{B_1}}\leq C_{N,p}\norm{u}_{\sob{B_1\setminus \overline{B_{1/2}}}}.$

\

\item[ii)] If $N=p+4p/(p-1)$, then \[\vert u(r)\vert \leq C_{p+4p/(p-1),p}\norm{u}_{\sob{B_1\setminus \overline{B_{1/2}}}}\left (\vert \log r\vert +1\right ),\;\forall r\in (0,1].\]

\

\item[iii)] If $N>p+4p/(p-1)$, then \[\vert u(r)\vert \leq C_{N,p}\norm{u}_{\sob{B_1\setminus \overline{B_{1/2}}}}r^{{-\frac{1}{p}\left(N-2\sqrt{\frac{N-1}{p-1}}-p-2\right)}},\;\forall r\in (0,1].\]

\end{itemize}
\label{th_1}
\end{theorem}
\begin{theorem}
Let $N\geq p+4p/(p-1)$, $g:\re\longrightarrow\re$ be a locally Lipschitz function, and $u\in\sob{B_1}$ be a semi-stable radial solution of \eqref{mainequation} satisfying $u_r(r)<0$ for all $r\in (0,1)$. Then there exists a constant $C'_{N,p}$ depending only on $N$ and $p$ such that:
\begin{itemize}

\item[i)] If $g\geq 0$, then
\[
\vert u_r(r)\vert \leq C'_{N,p}\norm{\D u}_{\leb{p}{B_1\setminus B_{1/2}}}r^{{-\frac{1}{p}\left(N-2\sqrt{\frac{N-1}{p-1}}-2\right)}},\;\forall r\in (0,1/2].
\]

\

\item[ii)] If $g\geq 0$ is nondecreasing, then
\[
\vert u_{rr}(r)\vert \leq C'_{N,p}\norm{\D u}_{\leb{p}{B_1\setminus B_{1/2}}}r^{{-\frac{1}{p}\left(N-2\sqrt{\frac{N-1}{p-1}}+p-2\right)}},\;\forall r\in (0,1/2].
\]

\

\item[iii)] If $g\geq 0$ is nondecreasing and convex, then
\[
\vert u_{rrr}(r)\vert \leq C'_{N,p}\norm{\D u}_{\leb{p}{B_1\setminus B_{1/2}}}r^{{-\frac{1}{p}\left(N-2\sqrt{\frac{N-1}{p-1}}+2p-2\right)}},\;\forall r\in (0,1/2].
\]

\end{itemize}
\label{th_2}
\end{theorem}

\begin{remark}
We emphasize that the estimates obtained in
Theorems \ref{th_1} and \ref{th_2} are in terms of the
$W^{1,p}$ norm of the annulus $B_1\setminus \overline{B_{1/2}}$, while
$u$ is required to belong to $W^{1,p}(B_1)$. In fact, this requirement
is essential to obtain our results, since we can easily find
semi-stable radially decreasing solutions of \eqref{mainequation} (for instance $u(r)=r^{-s}-1$, with $s$ large enough), not in the energy class $W^{1,p}(B_1)$, for which the statements of Theorems
\ref{th_1} and \ref{th_2} fail to satisfy.
\end{remark}
\begin{remark}
To our knowledge there is no estimates of $\vert u_{rr}\vert$ or $\vert u_{rrr}\vert$ in the literature for this kind of solutions. Moreover, we prove that without assumptions on the sign of $g$, $g'$ or $g''$ it
is not possible to obtain any pointwise estimate for $\vert
u_r\vert$, $\vert u_{rr}\vert$ or $\vert u_{rrr}\vert$ (see
Corollaries \ref{nohay}, \ref{nohayy} and \ref{nohayyy}).
\end{remark}

\section{Proof of the main results}

\begin{lemma}
Let $N\geq p>1$, $g:\re\longrightarrow\re$ be a locally Lipschitz function, and $u\in\sob{B_1}$ be a semi-stable radial solution of \eqref{mainequation} satisfying $u_r(r)<0$ for all $r\in (0,1)$. Then there exists a constant $K_{N,p}$ depending only on $N$ and $p$ such that:
\begin{equation}
\int_0^r{\vert u_r(t)\vert^p t^{N-1}\,dt}\leq K_{N,p}\norm{\D u}^p_{\leb{p}{B_1\setminus B_{1/2}}}r^{2\sqrt{\frac{N-1}{p-1}}+2},\;\forall r\in [0,1].
\label{des_1}
\end{equation}
\label{Lemma_1}
\end{lemma}
\begin{proof}
Let us use \cite[Lem. 2.2]{MR2476421} (see also the proof of \cite[Lem. 2.3]{MR2476421}) to assure that
\begin{equation}
(N-1)\int_{B_1}{\vert u_r\vert^p\eta^2\,dx}\leq (p-1)\int_{B_1}{\vert u_r\vert^p\vert \D(\vert x\vert\eta)\vert^2\,dx},
\label{des_2}
\end{equation}
for every radial Lipschitz function $\eta$ vanishing on $\partial B_1$.

We now fix $r\in (0,1/2)$ and consider the function
\[
\eta(t)=\left\lbrace \begin{array}{ll}
r^{-\sqrt{\frac{N-1}{p-1}}-1}&\text{, if }0\leq t\leq r,\\
t^{-\sqrt{\frac{N-1}{p-1}}-1}&\text{, if }r< t\leq 1/2,\\
2^{\sqrt{\frac{N-1}{p-1}}+2}(1-t)&\text{, if }1/2< t\leq 1.
\end{array}\right.
\]
Let $v(t)=(N-1)\eta(t)^2-(p-1)(t\eta(t))'^2$. Since $v(t)=0$ for $r<t\leq 1/2$, inequality \eqref{des_2} shows that
\begin{eqnarray*}
(N-p)r^{-2\sqrt{\frac{N-1}{p-1}}-2}\int_0^r{\vert u_r(t)\vert^pt^{N-1}\,dt}&\leq& -\int_{1/2}^1{v(t)\vert u_r(t)\vert^pt^{N-1}\,dt}\\
&\leq&\alpha_{N,p}\int_{1/2}^1{\vert u_r(t)\vert^pt^{N-1}\,dt},
\end{eqnarray*}
where the constant $\alpha_{N,p}=\underset{1/2\leq t\leq 1}{\max}{-v(t)}$ dependes only on $N$ and $p$. This establishes \eqref{des_1} for $r\in [0,1/2]$, if $N>p$.

If $r\in (1/2,1]$ and $N>p$ then, applying the above inequality for $r=1/2$, we obtain
\begin{eqnarray*}
\int_0^r{\vert u_r(t)\vert^pt^{N-1}\,dt}\leq\int_0^{1/2}{\vert u_r(t)\vert^pt^{N-1}\,dt}+\int_{1/2}^1{\vert u_r(t)\vert^pt^{N-1}\,dt}\\
\leq \left [\frac{\alpha_{N,p}}{N-p}\left (\frac{1}{2}\right )^{2\sqrt{\frac{N-1}{p-1}}+2}+1\right ]\int_{1/2}^1{\vert u_r(t)\vert^pt^{N-1}\,dt}\\
\leq (2r)^{2\sqrt{\frac{N-1}{p-1}}+2}\left [\frac{\alpha_{N,p}}{N-p}\left (\frac{1}{2}\right )^{2\sqrt{\frac{N-1}{p-1}}+2}+1\right ]\int_{1/2}^1{\vert u_r(t)\vert^pt^{N-1}\,dt},
\end{eqnarray*}
which is the desired conclusion with $K_{N,p}=\frac{1}{\omega_N}\left (\frac{\alpha_{N,p}}{N-p}+2^{2\sqrt{\frac{N-1}{p-1}}+2}\right )$.

Finally, if $N = p$, changing the definition of $\eta(t)$ in $[0,r]$ by
\[
\eta(t)=\left\lbrace \begin{array}{ll}
\frac{r^{-\sqrt{\frac{N-1}{p-1}}}}{r_0}&\text{, if }0\leq t\leq r_0,\\
\\
\frac{r^{-\sqrt{\frac{N-1}{p-1}}}}{t}&\text{, if }r_0< t\leq r,
\end{array}\right.
\]
for arbitrary $r_0\in (0,r)$, we obtain
\[
(N-1)r^{-2\sqrt{\frac{N-1}{p-1}}-2}\int_{r_0}^r{\left (\frac{r}{t}\right )^2\vert u_r(t)\vert^pt^{N-1}\,dt}\leq C_{N,p}\int_{1/2}^1{\vert u_r(t)\vert^pt^{N-1}\,dt}.
\]
Letting $r_0\rightarrow 0$ and taking into account that $r/t\geq 1$ for $0<t\leq r$ yields \eqref{des_1} for $N=p$ and $r\in [0,1/2]$. If $r\in (1/2,1]$, we apply similar arguments to the case $N>p$ to complete the proof.
\end{proof}
\begin{proposition}
Let $N\geq p>1$, $g:\re\longrightarrow\re$ be a locally Lipschitz function, and $u\in\sob{B_1}$ be a semi-stable radial solution of \eqref{mainequation} satisfying $u_r(r)<0$ for all $r\in (0,1)$. Then there exists a constant $K'_{N,p}$ depending only on $N$ and $p$ such that:
\begin{equation}
\left\vert u(r)-u\left (\frac{r}{2}\right )\right\vert\leq K'_{N,p}\norm{\D u}_{\leb{p}{B_1\setminus B_{1/2}}}r^{{-\frac{1}{p}\left(N-2\sqrt{\frac{N-1}{p-1}}-p-2\right)}},\;\forall r\in (0,1].
\label{des_3}
\end{equation}
\label{prop_1}
\end{proposition}
\begin{proof}
Fix $r\in (0,1]$. Applying H\"{o}lder's inequality and Lemma \ref{Lemma_1} we deduce
\begin{eqnarray*}
\left \vert u(r)-u\left (\frac{r}{2}\right )\right \vert &=& \int_{r/2}^r{\vert u_r(t)\vert t^{\frac{N-1}{p}}\frac{1}{t^{\frac{N-1}{p}}}\,dt}\\
&\leq & \left (\int_{r/2}^r{\vert u_r(t)\vert^p t^{N-1}\,dt}\right )^{\frac{1}{p}}\left (\int_{r/2}^r{t^{-\frac{(N-1)p'}{p}}\,dt}\right )^{\frac{1}{p'}}\\
&\leq& K_{N,p}^{\frac{1}{p}}\norm{\D u}_{\leb{p}{B_1\setminus B_{1/2}}}r^{\frac{2}{p}\sqrt{\frac{N-1}{p-1}}+\frac{2}{p}}\left (r^{-\frac{(N-1)}{p-1}+1}\int_{1/2}^1{t^{-\frac{(N-1)}{p-1}}\,dt}\right )^{\frac{p-1}{p}},
\end{eqnarray*}
and \eqref{des_3} is proved.
\end{proof}
\begin{proof}[Proof of the Theorem \ref{th_1}.]
Let $0<r\leq 1$. Then, there exist $m\in\mathbb{N}$ and $1/2<r_1\leq 1$ such that $r=r_1/2^{m-1}$. Since $u$ is radial we have $u(r_1)\leq \Vert u\Vert_{\leb{\infty}{B_1\setminus \overline{B_{1/2}}}}\leq \gamma_{N,p}\Vert u\Vert_{\sob{B_1\setminus \overline{B_{1/2}}}}$, where $\gamma_{N,p}$ dependes only on $N$ and $p$. From this an Proposition \ref{prop_1}, it follows that
\begin{equation}
\begin{array}{rcl}
\vert u(r)\vert &\leq& \vert u(r)-u(r_1)\vert +\vert u(r_1)\vert\\
&=&\sum\limits_{i=1}^{m-1}{\left \vert u\left (\frac{r_1}{2^{i-1}}\right )-u\left(\frac{r_1}{2^i}\right )\right \vert }+\vert u(r_1)\vert\\
&\leq& K'_{N,p}\norm{\D u}_{\leb{p}{B_1\setminus B_{1/2}}}\sum\limits_{i=1}^{m-1}{\left (\frac{r_1}{2^{i-1}}\right )^{{-\frac{1}{p}\left(N-2\sqrt{\frac{N-1}{p-1}}-p-2\right)}}}\\&&+\gamma_{N,p}\Vert u\Vert_{\sob{B_1\setminus \overline{B_{1/2}}}}\\
&\leq& \left (K'_{N,p}\sum\limits_{i=1}^{m-1}{\left (\frac{r_1}{2^{i-1}}\right )^{{-\frac{1}{p}\left(N-2\sqrt{\frac{N-1}{p-1}}-p-2\right)}}}+\gamma_{N,p}\right )\Vert u\Vert_{\sob{B_1\setminus \overline{B_{1/2}}}}.\\
\end{array}
\label{des_4}
\end{equation}
\begin{itemize}
\item If $p\leq N<p+4p/(p-1)$, we have ${-\frac{1}{p}\left(N-2\sqrt{\frac{N-1}{p-1}}-p-2\right)}>0$. Then
\[
\sum\limits_{i=1}^{m-1}{\left (\frac{r_1}{2^{i-1}}\right )^{{-\frac{1}{p}\left(N-2\sqrt{\frac{N-1}{p-1}}-p-2\right)}}}\leq \sum\limits_{i=1}^{\infty}{\left (\frac{1}{2^{i-1}}\right )^{{-\frac{1}{p}\left(N-2\sqrt{\frac{N-1}{p-1}}-p-2\right)}}},
\]
which is a convergent series.
\item If $N=p+4p/(p-1)$, we have ${-\frac{1}{p}\left(N-2\sqrt{\frac{N-1}{p-1}}-p-2\right)}=0$. From \eqref{des_4} we obtain
\begin{eqnarray*}
\vert u(r)\vert&\leq& \left [K'_{N,p} (m-1)+\gamma_{N,p}\right ]\Vert u\Vert_{\sob{B_1\setminus \overline{B_{1/2}}}}\\
&=& \left [K'_{N,p} \left (\frac{\log{r_1}-\log{r}}{\log{2}}\right )+\gamma_{N,p}\right ]\Vert u\Vert_{\sob{B_1\setminus \overline{B_{1/2}}}}\\
&\leq& \left ( \frac{K'_{N,p}}{\log{2}}+\gamma_{N,p}\right )(\vert \log{r}\vert +1)\Vert u\Vert_{\sob{B_1\setminus \overline{B_{1/2}}}},\\
\end{eqnarray*}
\item If $N>p+4p/(p-1)$, we have ${-\frac{1}{p}\left(N-2\sqrt{\frac{N-1}{p-1}}-p-2\right)}<0$. Then
\[
\sum\limits_{i=1}^{m-1}{\left (\frac{r_1}{2^{i-1}}\right )^{{-\frac{1}{p}\left(N-2\sqrt{\frac{N-1}{p-1}}-p-2\right)}}}=\frac{r^{{-\frac{1}{p}\left(N-2\sqrt{\frac{N-1}{p-1}}-p-2\right)}}-r_1^{{-\frac{1}{p}\left(N-2\sqrt{\frac{N-1}{p-1}}-p-2\right)}}}{(1/2)^{{-\frac{1}{p}\left(N-2\sqrt{\frac{N-1}{p-1}}-p-2\right)}}-1}.
\]
From \eqref{des_4}, we conclude
\[
\vert u(r)\vert\leq\left ( \frac{K'_{N,p}}{(1/2)^{{-\frac{1}{p}\left(N-2\sqrt{\frac{N-1}{p-1}}-p-2\right)}}-1}+\gamma_{N,p}\right )r^{{-\frac{1}{p}\left(N-2\sqrt{\frac{N-1}{p-1}}-p-2\right)}}\Vert u\Vert_{\sob{B_1\setminus \overline{B_{1/2}}}},
\]
\end{itemize}
which completes the proof.
\end{proof}
\begin{lemma}\label{lemma23}
Let $N\geq 1$, $p>1$, $g:\re\longrightarrow\re$ be a nonnegative and nondecreasing locally Lipschitz function, and $u\in\sob{B_1}$ be a semi-stable radial solution of \eqref{mainequation} such that $u_r<0$ for all $r\in (0,1)$. Then
%
\begin{equation}
g(u(r))\leq N\frac{\vert u_r\vert^{p-1}}{r},\ \ \forall r\in (0,1].
\label{ineq_g}
\end{equation}
Moreover, if $g$ convex then
\begin{equation}
g'(u(r))\leq M_{N,p}\frac{\vert u_r(r)\vert^{p-2}}{r^2},\ \ \forall r\in (0,1],
\label{ineq_dg}
\end{equation}
where $M_{N,p}$ is a constant depending only on $N$ and $p$.

\label{lemma_g}
\end{lemma}
\begin{proof}
Consider the function
\begin{equation}
\Psi(r):=-Nr^{1-1/N}\left \vert u_r\left (r^{1/N}\right )\right \vert^{p-2}u_r\left (r^{1/N}\right ),\,r\in (0,1].
\label{function_psi}
\end{equation}
It is easy to check to that $\Psi'(r)=g\left (u\left (r^{1/N}\right )\right )$, $r\in (0,1]$. As $g$ in nonnegative and nondecreasing we have that $\Psi$ is a nonnegative nondecreasing concave function. It follows immediately that
\begin{equation}
0\leq\Psi'(r)\leq\Psi(r)/r,\,r\in (0,1],
\label{ineq_psi}
\end{equation}
and we obtain \eqref{ineq_g}.

To obtain $ii)$, we first observe that from \eqref{mainequation} it is obtained
$$
u_{rr}=-\frac{1}{p-1}\left( \frac{g(u)}{\vert u_r\vert^{p-2}}+\frac{N-1}{r}u_r\right),\ \ \forall r\in (0,1].
$$
Therefore, using the nonnegativeness of $g$ and \eqref{ineq_g} we deduce that

\begin{equation}\label{cota_u_rr}
\vert u_{rr}\vert\leq \frac{1}{p-1}\left(\frac{g(u)}{\vert u_r\vert^{p-2}}+\frac{N-1}{r}\vert u_r\vert\right)\leq \left(\frac{2N-1}{p-1}\right)\frac{\vert u_r\vert}{r},\ \ \forall r\in (0,1].
\end{equation}

For fixed $\alpha\in\mathbb{R}$ an easy computation shows that
\begin{eqnarray*}
\partial_r\left (r^\alpha\vert u_r\vert^{p-2}\right )&=&\alpha r^{\alpha -1}\vert u_r\vert^{p-2}-(p-2)r^\alpha u_{rr}\vert u_r\vert^{p-3}\\
&\geq & r^{\alpha -1}\vert u_r\vert^{p-2}\left(\alpha-\frac{\vert p-2\vert (2N-1)}{p-1}\right),\ \ \forall r\in (0,1].
\end{eqnarray*}

Thus $r^\alpha\vert u_r\vert^{p-2}$ is nondecreasing for $\alpha=\frac{\vert p-2\vert (2N-1)}{p-1}$. Using this, the monotocity of $g'(u(r))$ and the semi-stability of $u$, we deduce that

\begin{eqnarray*}
g'(u(r))\int_0^rs^{N-1}\xi(s)^2ds&\leq&\int_0^rs^{N-1}g'(u(s))\xi(s)^2ds\\
&\leq& (p-1)\int_0^r\vert u_r(s)\vert^{p-2}s^\alpha s^{N-1-\alpha}\xi'(s)^2ds\\
&\leq& (p-1)\vert u_r(r)\vert^{p-2}r^\alpha \int_0^r s^{N-1-\alpha}\xi'(s)^2ds,
\end{eqnarray*}

\noindent for every $r\in (0,1)$ and every $\xi\in C^1$ with compact support in $(0,r)$.

Taking $\xi(s)=\zeta(\frac{s}{r})$ for $s\in [0,r]$, where $\zeta\in C^1$ is any  function with compact support in $(0,1)$, we obtain \eqref{ineq_dg}.
\end{proof}
\begin{proof}[Proof of the Theorem \ref{th_2}]
\begin{itemize}
\item[]
\item[i)] We first observe that $\partial_r\left (-r^{N-1}\vert u_r\vert^{p-2}u_r\right )=r^{N-1}g(u)$. Hence $-r^{N-1}\vert u_r\vert^{p-2}u_r$ is positive nondecreasing function and so is $\left (-r^{N-1}\vert u_r\vert^{p-2}u_r\right )^{\frac{p}{p-1}}$. Thus, for $0<r\leq 1/2$, we have
\begin{eqnarray*}
\int_0^{2r}{\vert u_r(t)\vert^pt^{N-1}\,dt}&\geq&\int_r^{2r}{\vert u_r(t)\vert^pt^{N-1}\,dt}\\
&=&\int_r^{2r}{\left (-t^{N-1}\vert u_r\vert^{p-2}u_r\right )^{\frac{p}{p-1}}t^{N-\frac{p(N-1)}{p-1}-1}\,dt}\\
&\geq&r^{\frac{p(N-1)}{p-1}}\vert u_r(r)\vert^p\int_r^{2r}{t^{N-\frac{p(N-1)}{p-1}-1}\,dt}\\
&=&r^N\vert u_r(r)\vert^p\int_1^{2}{t^{-\frac{N-1}{p-1}}\,dt},\\
\end{eqnarray*}
from this and Lemma \ref{Lemma_1} we obtain $i)$.
\item[ii)] 
Since \eqref{cota_u_rr} and $i)$ it follows $ii)$.
\item[iii)] From \eqref{mainequation} we obtain
\[
u_{rrr}=-\frac{1}{p-1}\left (\frac{g'(u)u_r}{\vert u_r\vert^{p-2}}-(p-2)\frac{u_ru_{rr}g(u)}{\vert u_r\vert^p}-\frac{N-1}{r^2}u_r+\frac{N-1}{r}u_{rr}\right ),
\]
\noindent for every $r\in (0,1)$. Therefore from \eqref{ineq_g}, \eqref{ineq_dg} and \eqref{cota_u_rr}, we obtain
\begin{eqnarray*}
\vert u_{rrr}\vert&\leq&\frac{1}{p-1}\left (\frac{g'(u)\vert u_r\vert}{\vert u_r\vert^{p-2}}+\vert p-2\vert\frac{\vert u_r\vert\vert u_{rr}\vert g(u)}{\vert u_r\vert^p}+\frac{N-1}{r^2}\vert u_r\vert+\frac{N-1}{r}\vert u_{rr}\vert\right )\\
&\leq&\frac{1}{p-1}\left (M_{N,p}+\frac{N(2N-1)\vert p-2\vert}{p-1}+(N-1)+\frac{(N-1)(2N-1)}{p-1}\right )\frac{\vert u_r\vert}{r^2},
\end{eqnarray*}
\noindent for every $r\in (0,1]$, and $iii)$ follows from $i)$.
\end{itemize}
\end{proof}

\begin{lemma} Let $N\geq 1$, $p>1$, $g:\re\longrightarrow\re$ be a locally Lipschitz nonnegative and nondecreasing function and $u$ be a radial solution of \eqref{mainequation} satisfying $u_r(r)<0$ for all $r\in (0,1)$. Then
\begin{itemize}
\item[i)] $r^{N-1}\vert u_r\vert^{p-1}$ is nondecreasing for $r\in (0,1]$.
\item[ii)] $r^{-1}\vert u_r\vert^{p-1}$  is nonincreasing for $r\in (0,1]$.
\item[iii)] $\max_{t\in[1/2,1]}{\vert u_r(t)\vert}\leq 2^\frac{N}{p-1}\min_{t\in[1/2,1]}{\vert u_r(t)\vert}$.
\item[iv)] $\norm{\D u}_{\leb{p}{B_1\setminus B_{1/2}}}\leq q_{N,p}\min_{t\in[1/2,1]}{\vert u_r(t)\vert}$ for a certain constant $q_{N,p}$ depending only on $N$ and $p$.
\end{itemize}
\label{lem_sol_ext}
\end{lemma}
\begin{proof}\
\begin{itemize}
\item[i)] Since $u_r<0$ we have $\partial_r\left (r^{N-1}\vert u_r\vert^{p-1}\right )=r^{N-1}g(u)\geq 0$
\item[ii)] As we have observed in the proof of Lemma \ref{lemma23}, the function $\Psi(r)=-Nr^{1-1/N}\left \vert u_r\left (r^{1/N}\right )\right \vert^{p-2}u_r\left (r^{1/N}\right )$ is nonnegative, nondecreasing and concave for $r\in(0,1]$. Therefore $\Psi(r)/r$ is nonincreasing, and $ii)$ follows immediately.
\item[iii)] Take $r_1,r_2\in [1/2,1]$ such that $\vert u_r(r_1)\vert =\min_{t\in[1/2,1]}{\vert u_r (t)\vert}$ and $\vert u_r(r_2)\vert =\max_{t\in[1/2,1]}{\vert u_r (t)\vert}$.\\
-If $r_2 \leq r_1$, we deduce from $i)$ that $\vert u_r(r_2 )\vert^{p-1}\leq (r_1/r_2)^{N-1}\vert u_r (r_1 )\vert^{p-1}\leq 2^{N}\vert u_r (r_1 )\vert^{p-1}$.\\
-If $r_2 > r_1$, we deduce from $ii)$ that $\vert u_r(r_2 )\vert^{p-1}\leq (r_2/r_1)\vert u_r (r_1 )^{p-1}\vert\leq 2\vert u_r (r_1 )\vert^{p-1}\leq 2^{N}\vert u_r (r_1 )\vert^{p-1}$.
\item[iv)] We see at once that
\[
\norm{\D u}_{\leb{p}{B_1\setminus B_{1/2}}}\leq \vert {B_1\setminus B_{1/2}}\vert^{1/p}\max_{t\in[1/2,1]}{\vert u_r (t)\vert}
\]

\end{itemize}
\end{proof}
\begin{proof}[Proof of Theorem \ref{th_ext}]
As we have mentioned, it is well known that $u^*$ is a semi-stable radially decreasing $\sob{B_1}$ solution of \eqref{mainequation} for $g(s)=\lambda f(s)$. Hence, we can apply to $u^*$ the results obtained in Theorems \ref{th_1} and \ref{th_2} and Lemma \ref{lem_sol_ext}.

Let us first prove $i)$, $ii)$, and $iii)$ for $r\in(0,1/2)$. Since $u^*(1)=0$, and on account of statement $iv)$ of Lemma \ref{lem_sol_ext}, we have
\[
\Vert u^*\Vert_{\sob{B_1\setminus \overline{B_{1/2}}}}\leq h_{N,p}\norm{\D u^*}_{\leb{p}{B_1\setminus B_{1/2}}}\leq h'_{N,p}\min_{t\in[1/2,1]}{\vert u^*_r (t)\vert},
\]
for certain constants $h_{N,p},h'_{N,p}$ depending only on $N$ and $p$. From this and Theorem \ref{th_1}:

\begin{itemize}
\item[i)] follows from the inequality $1\leq 2(1-r)$, for $r\in(0,1/2)$.
\item[ii)] follows from the inequality $\vert \log{r}\vert+1 \leq \left(\frac{\log{2}+1}{\log{2}}\right ) \vert \log{r}\vert$, for $r\in(0,1/2)$.

\item[iii)] follows from the inequality
\[
r^{{-\frac{1}{p}\left(N-2\sqrt{\frac{N-1}{p-1}}-p-2\right)}} \leq\left (\frac{(1/2)^{{-\frac{1}{p}\left(N-2\sqrt{\frac{N-1}{p-1}}-p-2\right)}}}{(1/2)^{{-\frac{1}{p}\left(N-2\sqrt{\frac{N-1}{p-1}}-p-2\right)}}-1},
\right )\left (r^{{-\frac{1}{p}\left(N-2\sqrt{\frac{N-1}{p-1}}-p-2\right)}}-1\right )
\]
\end{itemize}

\noindent for $r\in (0,1/2)$.

We next show $i)$, $ii)$, and $iii)$ for $r\in[1/2,1]$. From statement $iii)$ of Lemma \ref{lem_sol_ext} it follows that
\[
u^*(r)=\int_r^1{\vert u^*(t)\vert\,dt}\leq(1-r)2^\frac{N}{p-1}\min_{t\in[1/2,1]}{\vert u^*_r(t)\vert},\,\forall r\in[1/2,1],
\]
which is the desired conclusion if $N\leq p+4p/(p-1)$. If $N=p+4p/(p-1)$, our claim follows from the inequality $1-r\leq\vert \log{r}\vert$, for $r\in[1/2,1]$. Finally, if $N>p+4p/(p-1)$, the desired conclusion follows immediately from the inequality $1-r\leq z_{N,p}\left (r^{{-\frac{1}{p}\left(N-2\sqrt{\frac{N-1}{p-1}}-p-2\right)}}-1\right )$, for $r\in[1/2,1]$ and certain constant $z_{n,p}>0$.

We now prove statement $iv)$. In the case $k=1$ and $r\in(0,1/2)$, it follows immediatety from statement $i)$ of Theorem \ref{th_2} and statement $iv)$ of Lemma \ref{lem_sol_ext}. The case $k=1$ and $r\in[1/2,1]$ is also obvious on account of statement $iii)$ of Lemma \ref{lem_sol_ext} and inequality $1\leq r^{{-\frac{1}{p}\left(N-2\sqrt{\frac{N-1}{p-1}}-p-2\right)}}$, for $r\in[1/2,1]$, for $N\geq p+4p/(p-1)$.
Finally, as in the proof of statement $ii)$ and $iii)$ for $f$ convex, of Theorem \ref{th_2}, we have
\[
\vert u^*_{rr}(r)\vert\leq\left (\frac{2N-1}{p-1}\right)\frac{\vert u^*_{r}(r)\vert}{r},
\]
and
\[
\vert u^*_{rrr}(r)\vert\leq s_{N,p}\frac{\vert u^*_{r}(r)\vert}{r^2},
\]
for $r\in(0,1]$ and certain constant $s_{n,p}>0$., which gives statement $iv)$ and $v)$ from the case $k=1$.

\end{proof}
\section{A family of semi-stable solutions}
\begin{theorem}
Let $h\in (C^2\cap L^1 )(0, 1]$ be a nonnegative function and consider
\[
\Phi(r)=r^{2\sqrt{\frac{N-1}{p-1}}}\left (1+\int_0^r{h(s)\,ds}\right )\;\forall r\in(0,1].
\]
Define $u_r < 0$ by
\[
\Phi'(r)=(N-1)r^{N-3}\vert u_r (r)\vert^p\; \forall r \in (0, 1].
\]
Then, for $N \geq p+4p/(p-1)$, $u$ is a semi-stable radially decreasing unbounded $\sob{B_1}$ solution of a problem of the type \eqref{mainequation}, where $u$ is any function with radial derivative $u_r$.
\label{theorem_family}
\end{theorem}

To prove this theorem, we will use the following lemma, which is a generalization of the classical Hardy inequality:
\begin{lemma}{\cite{MR2885956}}
Let $\Phi\in C^1(0, L)$, $0 < L \leq\infty$, satisfying $\Phi' > 0$. Then
\[
\int_0^L{\frac{4\Phi^2}{\Phi'}\xi'^2}\geq\int_0^L{\Phi'\xi^2},
\]
for every $\xi\in C_c^{\infty}(0, L)$.
\label{lemma_gen_hardy}
\end{lemma}
\begin{proof}[Proof of the Theorem \ref{th_1}.]
 First of all, since $\Phi\in C^1(0, 1]\cap C[0, 1]$ is an increasing function, we obtain $\Phi'\in L^1(0, 1) $ and hence $r^{N-1}\vert u_r (r)\vert^p=r^2\Phi'(r)/(N-1)\in L^1(0, 1)$, which gives $u\in \sob{B_1}$.

On the other hand, since $\Phi'(r)\geq 2\sqrt{\frac{N-1}{p-1}} r^{2\sqrt{\frac{N-1}{p-1}}-1}$, $r\in(0,1]$, we deduce
\[
\vert u_r(r)\vert\geq \left (\frac{2}{N-1}\sqrt{\frac{N-1}{p-1}}\right )^{\frac{1}{p}}r^{{-\frac{1}{p}\left(N-2\sqrt{\frac{N-1}{p-1}}-2\right)}}.
\]
As $N\geq p+4p/(p-1)$, we have ${-\frac{1}{p}\left(N-2\sqrt{\frac{N-1}{p-1}}-2\right)}\leq -1$. It follows that $u_r\not\in L^1(0,1)$, and since $u$ is radially decreasing, we obtain $\lim_{r\rightarrow 0}{u(r)}=+\infty$.

Since $h\in C^2(0,1]$, it follows that $u_r \in C^2(0, 1]$. Therefore, $\Delta_p u\in C^1\left (\overline{B_1}\setminus\{0\} \right )$. Hence, taking  $g\in C^1(\re)$ such that $g(s)=-\Delta_p u(u^{-1}(s))$, for $s\in [u(1),+\infty)$, we conclude that $u$ is solution of a problem of the type \eqref{mainequation}.

It remains to prove that $u$ is semi-stable. Taking into account that $u_r\neq 0$ in $(0, 1]$ and applying \cite[Lem. 2.2]{MR2476421}, the semi-stability of $u$ is equivalent to
\begin{equation}
(p-1)\int_0^1{r^{N-1}\vert u_r\vert^p \xi'^2\,dx}\geq (N-1)\int_0^1{r^{N-3}\vert u_r\vert^p\xi^2\,dx},
\label{des_5}
\end{equation}
for every $\xi\in C_c^{\infty}(0,1)$.

For this purpose, we will apply the Lemma \ref{lemma_gen_hardy} above. From the definition of $\Phi$, it is easily seen that $\Phi'(r)\geq 2\sqrt{\frac{N-1}{p-1}}\frac{\Phi}{r}$, $r\in(0,1]$ It follows that
\[
\left (\frac{p-1}{N-1}\right )r^2\Phi'\geq\frac{4\Phi^2}{\Phi'}\text{ in }(0,1].
\]
Finally, since $\Phi'(r)=(N-1)r^{N-3}\vert u_r (r)\vert^p$, we deduce \eqref{des_5} by applying Lemma \ref{lemma_gen_hardy}.
\end{proof}

\begin{proposition}
Let $\{r_n\}\subset (0,1]$, $\{M_n\}\subset\re^+$ two sequences with $r_n\downarrow 0$. Then, for $N \geq p+4p/(p-1)$, there exists $u\in \sob{B_1}$, which is a semi-stable radially decreasing unbounded solution of a problem of the type \eqref{mainequation}, satisfying
\[
\vert u_r(r_n)\vert\geq M_n\;\;\forall n\in\mathbb{N}.
\]
\end{proposition}
\begin{proof}
It is easily seen that for every sequences $\{r_n\}\subset (0,1]$, $\{y_n\}\subset\re^+$, with $r_n\downarrow 0$, there exists a nonnegative function $h\in(C^2\cap L^1 )(0, 1]$ satisfying $h(r_n)=y_n$. Take $y_n=(N-1)M_n^pr_n^{N-2\sqrt{\frac{N-1}{p-1}}-3}$ and apply Theorem \ref{theorem_family} with this function $h$. It is clear, from the definition of $\Phi$, that $\Phi'(r)\geq r^{2\sqrt{\frac{N-1}{p-1}}}h(r)$, $r\in (0,1]$. Hence
\[
(N-1)r_n^{N-3}\vert u_r (r_n)\vert^p=\Phi'(r_n)\geq r_n^{2\sqrt{\frac{N-1}{p-1}}}h(r_n)=r_n^{2\sqrt{\frac{N-1}{p-1}}}(N-1)M_n^pr_n^{N-2\sqrt{\frac{N-1}{p-1}}-3},
\]
and the proposition follows.
\end{proof}
\begin{corollary}\label{nohay}
Let $N\geq p+4p/(p-1)$. There does not exist a function $\psi : (0, 1]\rightarrow \re^+$ with the following property: for every $u\in \sob{B_1}$ semi-stable radially decreasing solution of a problem of the type \eqref{mainequation}, there exist $C > 0$ and $\epsilon\in (0, 1]$ such that $\vert u_r (r)\vert\leq C\psi (r)$ for $r\in (0, \epsilon]$.
\label{corollary_u_r}
\end{corollary}
\begin{proof}
Suppose that such a function $\psi$ exists and consider the sequences $r_n = 1/n$, $M_n = n \psi(1/n)$. By the proposition above, there exists $u\in \sob{B_1}$, which is a semi-stable radially decreasing unbounded solution of a
problem of the type \eqref{mainequation}, satisfying $\vert u_r (1/n)\vert \geq n \psi (1/n)$, a contradiction.
\end{proof}

\begin{proposition}
Let $\{r_n\}\subset (0,1]$, $\{M_n\}\subset\re^+$ two sequences with $r_n\downarrow 0$. Then, for $N \geq p+4p/(p-1)$, there exists $u\in \sob{B_1}$, which is a semi-stable radially decreasing unbounded solution of a problem of the type \eqref{mainequation} with $g\geq 0$, satisfying
\[
\vert u_{rr}(r_n)\vert\geq M_n\;\;\forall n\in\mathbb{N}.
\]
\label{proposition_u_rr}
\end{proposition}

\begin{proof}
Let $h\in C^2(0, 1]$, increasing, satisfying $0 \leq h\leq 1$. Define $\Phi$ and $u_r$ as in Theorem \ref{theorem_family}. We claim that\\
\begin{itemize}
\item[i)] $u$ is a semi-stable radially decreasing unbounded $\sob{B_1}$ solution of a problem of the type \eqref{mainequation} with $g\geq  0$.
\item[ii)] $\vert u_r(r)\vert \leq D_{N,p}r^{{-\frac{1}{p}\left(N-2\sqrt{\frac{N-1}{p-1}}-2\right)}} $, $\forall r\in (0, 1]$, where $D_{N,p}$ only depends on $N$ and $p$.
\item[iii)] $-u_{rr}(r)\geq E_{N,p}r^{{-\frac{1}{p}\left(N-2\sqrt{\frac{N-1}{p-1}}-p-2\right)}}h'(r)-F_{N,p}r^{{-\frac{1}{p}\left(N-2\sqrt{\frac{N-1}{p-1}}+p-2\right)}}$, $\forall r\in (0,1]$, where $E_{N,p} > 0$ and $F_{N,p}$ only depend on $N$ and $p$.
\end{itemize}

Since $h$ is positive and increasing, then $\Phi'' > 0$. Hence $(N-1)r^{N-3}\vert u_r (r)\vert^p$ is increasing and so is
\[
r^{\frac{p(N-1)}{p-1}}\vert u_r(r)\vert^p=\left (-r^{N-1}\vert u_r(r)\vert^{p-2}u_r(r)\right )^{\frac{p}{p-1}}.
\]
This implies that $-r^{N-1}\vert u_r(r)\vert^{p-2}u_r(r)$ is increasing, which is is equivalent to the positiveness of $g$.

On account of $0 \leq h \leq 1$, we have $\Phi'(r)\leq G_{N,p}r^{2\sqrt{\frac{N-1}{p-1}}-1}$ in $(0, 1]$, for a constant $G_{N,p}$ that only depends on $N$ and $p$. Hence, from the definition of $u_r$, we obtain ii).

From the positiveness of $h$ and $N\geq p+4p/(p-1)$, we obtain $\Phi''(r)\geq r^{2\sqrt{\frac{N-1}{p-1}}}h'(r)$ in $(0,1]$. On the other hand, from the definition of $u_r$ we have $\Phi''(r)=(N-1)\left [(N-3)r^{N-4}\vert u_r(r)\vert^p+pr^{N-3}\vert u_r(r)\vert^{p-2}u_r(r)u_{rr}(r)\right ]$. Therefore, by $ii)$ and the previous inequality, we obtain iii).

Finally, it is easily seen that for every sequences $\{r_n\}\subset (0, 1]$, $\{y_n\}\subset\re^+$, with $r_n\downarrow 0$, there exists $h\in C^2(0, 1]$, increasing, satisfying $0 \leq h \leq 1$ and $h'(r_n)=y_n$. Take $y_n$ such that
\[
E_{N,p}r_n^{{-\frac{1}{p}\left(N-2\sqrt{\frac{N-1}{p-1}}-p-2\right)}}y_n-F_{N,p}r_n^{{-\frac{1}{p}\left(N-2\sqrt{\frac{N-1}{p-1}}+p-2\right)}}=M_n.
\]
Applying $iii)$ we deduce $-u_{rr} (r_n ) \geq M_n$ and the proof is complete.
\end{proof}

\begin{corollary}\label{nohayy}
Let $N\geq p+4p/(p-1)$. There does not exist a function $\psi : (0, 1]\rightarrow \re^+$ with the following property: for every $u\in \sob{B_1}$ semi-stable radially decreasing solution of a problem of the type \eqref{mainequation} with $g\geq 0$, there exist $C > 0$ and $\epsilon\in (0, 1]$ such that $\vert u_{rr} (r)\vert\leq C\psi (r)$ for $r\in (0, \epsilon]$.
\label{corollary_u_rr}
\end{corollary}
\begin{proof}
Arguing as in Corollary \ref{corollary_u_r} and using Proposition \ref{proposition_u_rr}, we conclude the proof of the corollary.
\end{proof}
\begin{proposition}
Let $\{r_n\}\subset (0,1]$, $\{M_n\}\subset\re^+$ two sequences with $r_n\downarrow 0$. Then, for $N \geq p+4p/(p-1)$, there exists $u\in \sob{B_1}$, which is a semi-stable radially decreasing unbounded solution of a problem of the type \eqref{mainequation} with $g,\,g'\geq 0$, satisfying
\[
\vert u_{rrr}(r_n)\vert\geq M_n\;\;\forall n\in\mathbb{N}.
\]
\label{proposition_u_rrr}
\end{proposition}
\begin{lemma}
For any dimension $N\geq p+4p/(p-1)$, there exists $\epsilon_{N,p} > 0$ with the following property: for every $h\in C^2(0,1]\cap C^1[0,1]$ satisfying $h(0) = 0$, $0\leq h'\leq\epsilon_{N,p}$ and $h''\leq 0$, $u$ is a semi-stable radially decreasing unbounded $\sob{B_1}$ solution of a problem of the type \eqref{mainequation} with $g,\,g'\geq 0$ where $u_r$ is defined in terms of $h$ as in Theorem \ref{theorem_family}.
\end{lemma}
\begin{proof}
Similarly as in the proof of Proposition \ref{proposition_u_rr} (item $i)$), $h'\geq 0$ implies that $u$ is a semi-stable radially decreasing unbounded $\sob{B_1}$ solution of a problem of the type \eqref{mainequation} with $g\geq 0$.

On the other hand, from the definition of $\Phi$ and $u_r$ it follows easily that
\begin{eqnarray*}
u_r&=&-\left [(N-1)^{-1}r^{-N+3}\Phi'\right ]^{\frac{1}{p}}\\
&=&-r^{{-\frac{1}{p}\left(N-2\sqrt{\frac{N-1}{p-1}}-2\right)}}\left [\frac{2+2\int_0^r{h(s)\,ds}}{\sqrt{(N-1)(p-1)}}+\frac{rh(r)}{N-1}\right ]^{\frac{1}{p}}.
\end{eqnarray*}
Put this last expression in the form $u_r=-r^{{-\frac{1}{p}\left(N-2\sqrt{\frac{N-1}{p-1}}-2\right)}}\varphi(r)$, where $\varphi(r)$ (and of course $u_r$ ) depends on $h$. Now consider the set
\[
X=\left\lbrace h\in C^2(0,1]\cap C^1[0,1]:\,h(0)=0,\,0\leq h',\,h''\leq 0 \right\rbrace,
\]
and the norm $\Vert h\Vert_X=\Vert h'\Vert_{\leb{\infty}{0,1}}$. Taking $\Vert h\Vert_X\rightarrow 0$, we have
\[
\lim\limits_{\Vert h\Vert_X\rightarrow 0}{\varphi}=\left [\frac{2}{\sqrt{(N-1)(p-1)}}\right ]^{\frac{1}{p}},\;\lim\limits_{\Vert h\Vert_X\rightarrow 0}{\varphi'}=0,
\]
\begin{equation}
\lim\limits_{\Vert h\Vert_X\rightarrow 0}{\left (\varphi''-\frac{rh''}{p(N-1)\varphi^{p-1}}\right )}=0,
\label{eq_lemma_urrr}
\end{equation}
where all the limits are taken uniformly in $r \in(0, 1]$. On the other hand, it is easy to check that
\begin{eqnarray*}
\frac{g'(u)r^2}{p-1}&=&-\vert u_r\vert^{p-2}\left [(p-2)\left (\frac{ru_{rr}}{u_r}\right )^2+(N-1)\left (\frac{ru_{rr}}{u_r}\right )\right.\\
&&\left.+\left (\frac{r^2u_{rrr}}{u_r}\right )-\frac{N-1}{p-1}\right]\\
&=&-\vert u_r\vert^{p-2}\left [(p-2)\left (\frac{r\varphi'}{\varphi}\right )^2+\frac{r^2\varphi''}{\varphi}\right.\\
&&\left.+\frac{1}{p}\left ((2-p)N+4(p-1)\sqrt{\frac{N-1}{p-1}}+3p-4\right )\left (\frac{r\varphi'}{\varphi}\right )\right.\\
&&\left .-\frac{1}{p^2}\left (N-2\sqrt{\frac{N-1}{p-1}}-2\right  )\left  (N+2(p-1)\sqrt{\frac{N-1}{p-1}}-2\right )\right.\\
&&\left.-\frac{N-1}{p-1}\right].\\
\end{eqnarray*}
Hence, from \eqref{eq_lemma_urrr}, we can assert that, for $h\in X$  with small $\Vert h\Vert_X$, $r^2g'(u)\geq 0$ in $(0,1]$ and the lemma follows.
\end{proof}
\begin{proof}[Proof of Proposition \ref{proposition_u_rrr}]
We follow the notation used in the previous lemma. From \eqref{eq_lemma_urrr}, we deduce that
\[
\lim\limits_{\Vert h\Vert_X\rightarrow 0}{\left (r^{{\frac{1}{p}\left(N-2\sqrt{\frac{N-1}{p-1}}+2p-2\right)}} u_{rrr}+\frac{r^3h''}{p(N-1)\varphi^{p-1}}\right )}=\sigma,
\]
uniformly in $r\in(0, 1]$, where

\[\sigma=-2^{\frac{1}{p}}\frac{\left [{\left (-\frac{1}{p}\right )\left(N-2\sqrt{\frac{N-1}{p-1}}-2\right)}\right ]\left [{\left (-\frac{1}{p}\right )\left(N-2\sqrt{\frac{N-1}{p-1}}+p-2\right)}\right ]}{\left (\sqrt{(N-1)(p-1)}\right)^{\frac{1}{p}}}<0.
\]
Then, taking $\epsilon'_{N,p}>0$ sufficient small (possibly less than $\epsilon_{N,p}$), we have that
\[
r^{{\frac{1}{p}\left(N-2\sqrt{\frac{N-1}{p-1}}+2p-2\right)}} u_{rrr}\geq-\frac{r^3h''}{p(N-1){\left [\frac{2}{\sqrt{(N-1)(p-1)}}\right ]^{\frac{p-1}{p}}}}+\sigma-1,\,\forall r\in(0,1],
\]
for $\Vert h\Vert_X\leq \epsilon'_{N,p}$.

Finally, it is easily seen that for every sequences $\{r_n\}\subset (0, 1]$, $\{y_n\}\subset\re^+$, with $r_n\downarrow 0$, there exists $h\in X$, with $\Vert h\Vert_X\leq \epsilon'_{N,p}$, satisfying $h''(r_n)=-y_n$. Take $y_n$ such that
\[
r_n^{{-\frac{1}{p}\left(N-2\sqrt{\frac{N-1}{p-1}}+2p-2\right)}}M_n=\frac{r^3y_n}{p(N-1){\left [\frac{2}{\sqrt{(N-1)(p-1)}}\right ]^{\frac{p-1}{p}}}}+\sigma-1.
\]
Applying the above inequality, we obtain $u_{rrr}(r_n)\geq M_n$ and the proof is complete.
\end{proof}
\begin{corollary}\label{nohayyy}
Let $N\geq p+4p/(p-1)$. There does not exist a function $\psi : (0, 1]\rightarrow \re^+$ with the following property: for every $u\in \sob{B_1}$ semi-stable radially decreasing solution of a problem of the type \eqref{mainequation} with $g,\,g'\geq 0$, there exist $C > 0$ and $\epsilon\in (0, 1]$ such that $\vert u_{rrr} (r)\vert\leq C\psi (r)$ for $r\in (0, \epsilon]$.
\end{corollary}
\begin{proof}
Applying Proposition \ref{proposition_u_rrr}, this follows by the same method as in Corollaries \ref{corollary_u_r} and \ref{corollary_u_rr}.
\end{proof}

\end{document}